\documentclass[11pt, a4paper]{amsart}
\newtheorem{prop}{Proposition}[section]
\newtheorem{lema}[prop]{Lemma}
\newtheorem{teo}[prop]{Theorem}
\newtheorem{corolario}[prop]{Corollary}

\newtheorem{remark}[prop]{\sc Remark}

\usepackage[usenames]{color}

\newcommand{\supp}{\mbox{supp}}
\newcommand{\uno}{1\!\!1}

\title[Decomposition of valuations]{A Jordan-like decomposition theorem for valuations on star bodies}

\author{Pedro Tradacete}

\address{Mathematics Department\\ Universidad Carlos III de Madrid \\  28911 Legan\'es (Madrid). Spain.}
\email{ptradace@math.uc3m.es }

\thanks{Support of Spanish MINECO under grants MTM2012-31286 and MTM2013-40985 is gratefully acknowledged.}

\author{Ignacio Villanueva}

\address{Departamento de An\'alisis Matem\'atico \\
Facultad de Matem\'aticas \\ Universidad Complutense de Madrid \\
Madrid 28040}

\email{ignaciov@mat.ucm.es}

\thanks{Partially supported by grants MTM2014-54240-P, funded by MINECO and QUITEMAD+-CM, Reference: S2013/ICE-2801, funded by Comunidad de Madrid}

\begin{document}

\begin{abstract}
We show  that every radial continuous valuation $V:\mathcal S_0^n\rightarrow \mathbb R$ defined on the $n$-dimensional star bodies $\mathcal S_0^n$, and verifying $V(\{0\})=0$, can be decomposed as a sum $V=V^+-V^-$, where both $V^+$ and $V^-$ are positive radial continuous valuations on $\mathcal S_0^n$ with $V^+(\{0\})=V^-(\{0\})=0$.

As an application, we show that  radial continuous rotationally invariant valuations $V$ on $\mathcal S_0^n$  can be characterized as the applications on star bodies which can be written as  $$V(K)=\int_{S^{n-1}}\theta(\rho_K)dm,$$ where $\theta:[0,\infty)\rightarrow \mathbb R$ is a continuous function, $\rho_K$ is the radial function associated to $K$ and $m$ is the Lebesgue measure on $S^{n-1}$.

This completes recent work of the second named author, where an analogous result is proved for the case of {\em positive} radial continuous rotationally invariant valuations.
\end{abstract}



\maketitle

\section{Introduction}

This note is devoted to the study of valuations on star bodies. A valuation is a function $V$, defined on a class of sets, with the property that
$$
V(A\cup B)+V(A\cap B)=V(A)+V(B).
$$
As a generalization of the notion of measure, valuations have become a relevant area of study in Convex Geometry. In fact, this notion played a critical role in M. Dehn's solution to Hilbert's third problem, asking whether an elementary definition for volume of polytopes was possible. See, for instance, \cite{Lu1}, \cite{Lu2} and the references there included for a  broad vision of the field.

In \cite{Had}, H. Hadwiger provided a characterization of continuous rotationally and translationally invariant valuations on convex bodies as linear combinations of the quermassintegrals. In \cite{Alesker}, S. Alesker studied the  valuations on convex bodies which are only rotationally invariant.

Valuations on convex bodies belong to the Brunn-Minkowski Theory. This theory has been extended in several important ways, and in particular, to the dual Brunn-Minkowski Theory, where convex bodies, Minkowski addition and Hausdorff metric  are replaced by star bodies, radial addition and  radial metric, respectively. The dual Brunn-Minkowski theory, initiated in \cite{Lut_mv1}, has been broadly developed and successfully applied to several areas, such as integral geometry,  local theory of Banach spaces and geometric tomography (see \cite{DGP}, \cite{Gabook} for these and other applications). In particular, it played a key role in the solution of the Busemann-Petty problem \cite{Ga1}, \cite{Ga2}, \cite{Zh}.

D. A. Klain initiated in \cite{Klain96}, \cite{Klain97} the study of rotationally invariant valuations on a certain class of star sets, namely those whose radial function is $n$-th power integrable.

In \cite{Vi}, the second named author initiated the study of valuations on star bodies. In this note we continue this study showing that every continuous valuation $V:\mathcal S_0^n\longrightarrow \mathbb R$ on the $n$-dimensional star bodies can be decomposed as the difference of two positive continuous valuations.

The precise result is

\begin{teo}\label{main}
Let $V:\mathcal S_0^n\longrightarrow \mathbb R$ be a radial continuous valuation on the $n$-dimensional star bodies $\mathcal S_0^n$ such that $V(\{0\})=0$. Then, there exist two radial continuous valuations $V^+, V^-:\mathcal S_0^n\longrightarrow \mathbb R_+$ such that $V^+(\{0\})=V^-(\{0\})=0$ and such that
$$
V=V^+-V^-.
$$
Moreover, if $V$ is rotationally invariant, so are $V^+$ and $V^-$.
\end{teo}

With this structural result at hand, the study of  continuous valuations on star bodies reduces to the simpler case of positive continuous valuations.

As an application, we can complete the main result of \cite{Vi}. In that paper, {\em positive} rotationally invariant continuous valuations $V$ on the star bodies of $\mathbb R^n$, satisfying that $V(\{0\})=0$ are  characterized by an integral representation as in Corollary  \ref{representacion} below. The question of whether a similar description is valid for the case of real-valued (not necessarily positive or negative) continuous rotationally invariant valuations was left open.

Now, Theorem \ref{main} allows us to give a positive answer to this question:

\begin{corolario}\label{representacion}
Let $V:\mathcal S_0^n\longrightarrow \mathbb R$ be a rotationally invariant radial continuous valuation on the $n$-dimensional star bodies $\mathcal S_0^n$. Then, there exists a continuous function $\theta:[0,\infty) \longrightarrow \mathbb R$ such that, for every $K\in \mathcal S_0^n$,
$$V(K)=\int_{S^{n-1}} \theta(\rho_K(t)) dm(t),$$
where $\rho_K$ is the radial function of $K$ and $m$ is the Lebesgue measure on $S^{n-1}$ normalized so that $m(S^{n-1})=1$.

Conversely, let $\theta:\mathbb R^+\longrightarrow \mathbb R$ be a continuous function. Then the application  $V:\mathcal S_0^n \longrightarrow \mathbb R$ given by
$$V(K)=\int_{S^{n-1}} \theta(\rho_K(t)) dm(t)$$ is a radial continuous rotationally invariant valuation.
\end{corolario}

\smallskip

As in \cite{Vi}, the function $\theta$ in Corollary \ref{representacion} is nothing but $\theta(\lambda)=V(\lambda S^{n-1})$.

\smallskip

The rest of the paper is structured as follows:
In Section \ref{sectionnotation} we describe our notation and some known facts that we will need.
In Section \ref{results} we prove Theorem \ref{main} and Corollary \ref{representacion}

\section{Notation and known facts}\label{sectionnotation}

A set $L\subset \mathbb R^n$ is a {\em star set} if it contains the origin and every line through $0$ that meets $L$ does so in a (possibly degenerate) line segment. Let $\mathcal S^n$  denote the set of the star sets of $\mathbb R^n$.

Given $L\in \mathcal S^n$, we define its {\em radial function} $\rho_L$ by
$$
\rho_L(t)=  \sup \{c\geq 0 \, : \, ct\in L\},
$$
for each $t\in\mathbb R^n$. Clearly, radial functions are completely characterized by their restriction to $S^{n-1}$, the euclidean unit sphere in $\mathbb R^n$, so from now on we consider them defined on $S^{n-1}$.

A star set $L$ is called a {\em star body} if $\rho_L$ is continuous. Conversely, given a positive continuous function $f:S^{n-1}\longrightarrow \mathbb R^+=[0,\infty)$ there exists a star body $L_f$ such that $f$ is the radial function of $L_f$.
We denote by $\mathcal S_0^n$ the set of $n$-dimensional star bodies and we denote by $C(S^{n-1})^+$ the set of positive continuous functions on $S^{n-1}$.

Given two sets $K,L\in \mathcal S^n$, we define their {\em radial sum} as the star set $K\tilde{+}L$ whose radial function is $\rho_K+\rho_L$. Note that $K\tilde{+}L\in \mathcal S_0^n$ whenever $K,L\in \mathcal S_0^n$.

The dual analog for the Hausdorff metric of convex bodies is the so called {\em radial metric}, which is defined by
$$
\delta(K,L)=\inf\{\lambda\geq 0 : K\subset L\tilde{+} \lambda B_n, L\subset K\tilde{+} \lambda B_n\},
$$
where $B_n$ denotes the euclidean unit ball of $\mathbb R^n$. It is easy to check that
$$
\delta(K,L)=\|\rho_K-\rho_L\|_\infty.
$$

\smallskip

An application $V:\mathcal S_0^n\longrightarrow \mathbb R$ is a {\em valuation} if for any $K,L\in\mathcal S_0^n$,
$$
V(K\cup L)+V(K\cap L)=V(K)+ V(L).
$$
It is clear that a linear combination of valuations is a valuation.

\smallskip

Given two functions $f_1, f_2\in C(S^{n-1})^+$, we denote their maximum and minimum by
$$
(f_1\vee f_2)(t)=\max \{f_1(t), f_2(t)\},
$$
$$
(f_1\wedge f_2)(t)=\min \{f_1(t), f_2(t)\}.
$$

Given two star bodies $K,L$, both $K\cup L$ and $K\cap L$ are star bodies, and it is easy to see that
$$
\rho_{K\cup L}=\rho_K\vee \rho_L, \hspace{1cm} \rho_{K\cap L}=\rho_K\wedge \rho_L.
$$

With this notation, a valuation $V:\mathcal S_0^n\rightarrow \mathbb R$ induces a function $\tilde V:C(S^{n-1})^+\rightarrow \mathbb R$ given by
$$
\tilde V(f)=V(L_f),
$$
where $L_f$ is the star body whose radial function satisfies $\rho_{L_f}=f$. If $V$ is continuous, then  $\tilde V$ is continuous with respect to the $\|\cdot\|_\infty$ norm in $C(S^{n-1})^+$ and satisfies
$$
\tilde V(f)+\tilde V(g)=\tilde V(f\vee g)+\tilde V(f\wedge g)
$$
for every $f,g\in C(S^{n-1})^+$. Conversely, every such function $\tilde V$ induces a radial continuous valuation on $\mathcal S_0^n$.

Given  $A\subset S^{n-1}$, we denote the closure of $A$ by $\overline{A}$.
Given  a function $f:S^{n-1} \longrightarrow \mathbb R$, we define the support of $f$ by $$supp(f)=\overline{\{t\in S^{n-1} \mbox{ such that } f(t)\not = 0\}},$$ and for any set $G\subset S^{n-1}$, we will write $f\prec G$ if $\supp(f)\subset G$. Throughout, $\uno:S^{n-1}\longrightarrow \mathbb R$ denotes the function constantly equal to 1.

For completeness, we state now a result of \cite{Vi} which will be needed later.

\begin{lema}\label{split}\cite[Lemmata 3.3 and  3.4]{Vi}
Let $\{G_i: i\in I\}$ be a family of open subsets of $S^{n-1}$. Let $G=\cup_{i\in I} G_i$. Then, for every $i\in I$ there exists a function $\varphi_i: G \longrightarrow [0,1]$ continuous in $G$ verifying  $\varphi_i\prec G_i$ and such that $\bigvee_{i\in I} \varphi_i=\uno$ in $G$. Moreover,  let  $f\in C(S^{n-1})^+$ verify $f\prec G$. Then, for every $i\in I$, the function  $f_i=\varphi_if$ belongs to $ C(S^{n-1})^+$. Also,  $f_i\prec G_i$ and  $\bigvee_{i\in I} f_i=f$. In particular, for every $i\in I$,  $0\leq f_i\leq f$.
\end{lema}

\section{The results}\label{results}

To prove Theorem \ref{main} we will need to control the maximum value of $V$ on certain sets. The first step in this direction is to show that $V$ is {\em bounded on bounded sets}:

We say that a valuation $V:\mathcal S^n_0\longrightarrow \mathbb R$ is  bounded on bounded sets  if for every $\lambda>0$ there exists  a real number $K>0$ such that, for every star body $L\subset \lambda B_n$, $|V(L)|\leq K.$

Equivalently, $V$ is bounded on bounded sets if for every $\lambda>0$ there exists $K>0$ such that for every $f\in C(S^{n-1})^+$ with $\|f\|_\infty\leq \lambda$ we have $\tilde{V}(f)\leq K$.

\begin{lema}\label{l:bbs}
Every radial continuous valuation $V:\mathcal S^n_0\longrightarrow \mathbb R$ is bounded on bounded sets.
\end{lema}

\begin{proof}
We reason by contradiction. If the result is not true, there exists $\lambda>0$ and  a sequence $(f_i)_{i\in \mathbb N}\subset C(S^{n-1})^+$, with $\|f_i\|_\infty\leq \lambda$ for every $i\in \mathbb N$ and such that $|\tilde{V}(f_i)|\rightarrow +\infty$.

Consider the function $$\theta:\mathbb R^+\longrightarrow \mathbb R$$ defined by $$\theta(c)=\tilde{V}(c \uno).$$ The continuity of $\tilde{V}$ implies that $\theta$ is continuous. Therefore, $\theta$ is uniformly continuous on $[0,\lambda]$. In particular, it is bounded on that interval. Therefore, there exists $M>0$ such that, for every $c\in [0,\lambda]$, $$|\tilde{V}(c \uno)|\leq M.$$

We define  inductively two sequences $(a_j)_{j\in \mathbb N}, (b_j)_{j\in \mathbb N}\subset \mathbb R^+$: Define first $a_0=0$, $b_0=\lambda$. Consider $c_0=\frac{a_0+b_0}{2}$.

We note that  $$\tilde{V}(f_i\vee c_0\uno) +\tilde{V}(f_i\wedge c_0\uno)=\tilde{V}(f_i) + \tilde{V}(c_0\uno).$$

Since $|\tilde{V}(c_0\uno)|\leq M$ and $|\tilde{V}(f_i)|\rightarrow +\infty$, we know that there must exist an infinite set $\mathbb M_1\subset \mathbb N$ such that for $i\in \mathbb M_1$ either $|\tilde{V}(f_i\vee c_0\uno)|\rightarrow +\infty$ or $|\tilde{V}(f_i\wedge c_0\uno)|\rightarrow +\infty$ as $i$ grows to $\infty$. In the first case, we set $a_1=c_0$, $b_1=\lambda$ and $f^1_i=f_i\vee c_0\uno$. In the second case, we set $a_1=0$ and $b_1=c_0$ and $f^1_i=f_i\wedge c_0\uno$. Now we define $c_1=\frac{a_1+ b_1}{2}$ and proceed similarly.

Inductively, we construct two sequences $(a_j), (b_j)\subset \mathbb R^+$, a decreasing sequence of infinite subsets $\mathbb M_j\subset \mathbb N$, and sequences $(f^j_i)_{i\in \mathbb M_j}\subset C(S^{n-1})^+$ such that, for every $j\in\mathbb N$,
$$
|a_j-b_j|=\frac{\lambda}{2^j},
$$
and for every $i\in\mathbb M_j$, for every $t\in S^{n-1}$,
$$
a_j\leq f^j_i(t)\leq b_j,
$$
and with the property that
$$
\lim_{ i\rightarrow \infty} |\tilde{V}(f^j_i)|=+\infty.
$$
Passing to a further subsequence we may assume without loss of generality that, for every $i\in \mathbb N$,  $$|\tilde{V}(f_i^i)|\geq i.$$

Call $d=\lim_i a_i$. If we consider now the sequence $(f_i^i)_{i\in \mathbb N}\subset C(S^{n-1})^+$, we have that $$\|f_i^i-d\uno\|_\infty \rightarrow 0$$ but $$|\tilde{V}(f_i^i)|\geq i,$$
in contradiction to the continuity of $\tilde{V}$ at $d\uno$.

\end{proof}

We thank the anonymous referee of \cite{Vi} for suggesting a procedure very similar to this as an alternative reasoning to show a statement in that paper.

\smallskip

In the rest of this note we will repeatedly use the fact that $S^{n-1}$ is a compact metric space. We will write $d$ to denote the euclidean metric in $S^{n-1}$.

We need to define an additional concept for our next result:

Given any set $A\subset S^{n-1}$, and $\omega>0$, let us consider the {\em outer parallel band} around $A$ defined by
$$
A_\omega=\{t\in S^{n-1} : 0<d(t, A)<\omega\}.
$$
Note that, for every $A\subset S^{n-1}$ and $\omega>0$, $A_\omega$ is an open set.

In our next result we use the fact that $V$ is bounded on bounded sets to  control  $V$ on these bands.  In  \cite{FK} these outer parallel bands are called {\em rims}, and they are used for similar purposes to ours.

\begin{lema}\label{rims}
Let $V:\mathcal S^n_0\rightarrow \mathbb R$ be a radial continuous valuation. Let $A\subset S^{n-1}$ be any set and $\lambda\in \mathbb R^+$.
$$
\lim_{\omega\rightarrow 0} \sup\{|\tilde{V}(f)|: \, f\prec A_\omega, \, \|f\|_\infty\leq \lambda\}=0.
$$
\end{lema}

\begin{proof}
We reason by contradiction. Suppose the result is not true. Then there exist $A\subset S^{n-1}$, $\lambda \in \mathbb R^+$, $\epsilon>0$, a sequence $(\omega_i)_{i\in \mathbb N}\subset \mathbb R$  and a sequence $(f_i)_{i\in \mathbb N}\subset C(S^{n-1})^+$ such that $\lim_{i\rightarrow \mathbb N} \omega_i=0$ and, for every $i\in \mathbb N$,

\begin{itemize}
\item   $\omega_i>0$
\item  $f_i\prec A_{\omega_i}$
\item $\|f_i\|_\infty\leq \lambda$
\item $|\tilde{V}(f_i)|\geq \epsilon.$

\end{itemize}

Therefore, there exists an infinite subset $I\subset \mathbb N$ such that either  $\tilde{V}(f_i) >\epsilon$ for every $i\in I$ or $\tilde{V}(f_i) <\epsilon$ for every $i\in I$. So, we assume without loss of generality that $\tilde{V}(f_i) >\epsilon$ for every $i\in I$. The case $\tilde{V}(f_i) <\epsilon$ is totally analogous.

Consider $f_1$. Using the continuity of $\tilde{V}$ at $f_1$, we get the existence of $\delta>0$ such that for every  $g\in C(S^{n-1})^+$ with $\|f-g\|_\infty<\delta$,
$$
|\tilde{V}(f)-\tilde{V}(g)|\leq \frac{\epsilon}{2}.
$$

Since $f_1$ is uniformly continuous and $f_1(t)=0$ for every $t\in A\subset S^{n-1}\backslash A_{\omega_1}$, there exists $0<\rho<\omega_1$ such that, for every $t\in S^{n-1}$ with $d(t, A)<\rho$, $f_1(t)<\delta$.
We consider the disjoint closed sets
$$
C_1=\{t\in S^{n-1} : d(t, A)\leq \frac{\rho}{2}\}
$$
and
$$
C_2=f_1^{-1}\left([\delta, \lambda]\right).
$$
By Urysohn's Lemma, we can consider a continuous function $\psi_1$ with $\psi_{1|_{C_1}}=0$, $\psi_{1|_{C_2}}=1$ and  $0\leq \psi_1(t)\leq 1$ for every $t\in S^{n-1}$. We consider now the function $\psi_1 f_1\in C(S^{n-1})^+$. On the one hand, $\|f_1-\psi_1 f_1\|_\infty\leq \delta$ and, therefore, $$|\tilde{V}(\psi_1 f_1)|\geq \left||\tilde{V}(f_1)|-|\tilde{V}(f_1)-\tilde{V}(\psi_1 f_1)|\right|> \epsilon-\frac{\epsilon}{2}=\frac{\epsilon}{2}.$$

On the other hand,  $\psi_1 f_1\prec A_{\omega_1}\backslash A_{\frac{\rho}{2}}$. Now, we can choose $\omega_{i_2}< \frac{\rho}{2}$ and we can reason similarly as above with the function $f_{i_2}$.

Inductively, we construct a sequence of functions $(\psi_j f_{i_j})_{j\in \mathbb N}\subset C(S^{n-1})^+$ with disjoint support such that $\tilde{V}(\psi_j f_{i_j})>\frac{\epsilon}{2}$. Noting that
$$
\tilde{V}\left(\bigvee_j \psi_j f_{i_j}\right)=\sum_j \tilde{V}(\psi_j f_{i_j}),
$$
and that
$$
\|\bigvee_j \psi_j f_{i_j}\|_\infty\leq\lambda,
$$
we get a contradiction with the fact that $V$ is bounded on bounded sets.
\end{proof}

Now we can prove Theorem \ref{main}.

\begin{proof}[Proof of Theorem \ref{main}]
Let $V:\mathcal S_0^n\longrightarrow \mathbb R$ be as in the hypothesis. For every $f\in C(S^{n-1})^+$, we define
$$
\tilde{V}^+(f)=\sup\{\tilde{V}(g): \, 0\leq g \leq f\},
$$
and we consider the application $\tilde{V}:\mathcal S_0^n\longrightarrow \mathbb R$ defined by
$V^+(K)=\tilde{V}^+(\rho_K)$.

Assume for the moment that $V^+$ is a radial continuous valuation. In that case, the result follows easily:

First we note that it follows from $\tilde V(0)=0$ that $V^+(\{0\})=0$ and that, for every $f\in C(S^{n-1})^+$, one has   $\tilde V^+(f)\geq 0$. Therefore,  $V^+(K)\geq 0$ for every $K\in \mathcal S_0^n$.

We define next $V^-=V^+-V$. Clearly, $V^-$ is a radial continuous valuation and $V^-(\{0\})=0$.

By the definition of $ V^+$, it follows that, for every $K\in \mathcal S_0^n$, one has $V(K)\leq  V^+(K)$. Thus, $V^-(K)\geq 0$.

Now, trivially $$V=V^+-V^-.$$

In addition, if $V$ is rotationally invariant, then clearly $V^+$ and $V^-$ are so.

\smallskip

Therefore, we will finish if we show that
$V^+$ is  a radial continuous valuation. Let us prove it.

First, we see that  it is a valuation. Let $f_1, f_2\in C(S^{n-1})^+$. We have to check that
\begin{equation}\label{igualdad}
\tilde{V}^+(f_1\vee f_2)+ \tilde{V}^+(f_1\wedge f_2) = \tilde{V}^+(f_1)+ \tilde{V}^+(f_2).
\end{equation}

Fix $\epsilon>0$. We choose $0\leq g_1\leq f_1$ such that $\tilde{V}^+(f_1)\leq \tilde{V}(g_1)+\epsilon$, and $0\leq g_2\leq f_2$ such that $\tilde{V}^+(f_2)\leq \tilde{V}(g_2)+\epsilon$.

Then, $$\tilde{V}^+(f_1)+ \tilde{V}^+(f_2)\leq \tilde{V}(g_1)+ \tilde{V}(g_2)+ 2\epsilon = \tilde{V}(g_1\vee g_2) + \tilde{V}(g_1\wedge g_2) + 2\epsilon\leq $$ $$\leq \tilde{V}^+(f_1\vee f_2) + \tilde{V}^+(f_1\wedge f_2) + 2\epsilon, $$
where the last inequality follows from the fact that $0\leq g_1\vee g_2\leq f_1\vee f_2$, $0\leq g_1\wedge  g_2\leq f_1\wedge f_2$. Since $\epsilon>0$ was arbitrary, this proves one of the inequalities in \eqref{igualdad}.

For the other one, fix again $\epsilon>0$. We choose $0\leq g \leq f_1\vee f_2 $ such that $\tilde{V}^+(f_1\vee f_2)\leq \tilde{V}(g)+\epsilon$, and $0\leq h \leq f_1\wedge f_2 $ such that $\tilde{V}^+(f_1\wedge f_2)\leq \tilde{V}(h)+\epsilon$.

Let us consider the sets
$$
A=\{ t\in S^{n-1} : f_1(t)\geq f_2(t)\}
$$
and
$$
B=\{ t\in S^{n-1} : f_1(t)< f_2(t)\}.
$$

Let $\lambda=\|f_1\vee f_2\|_\infty$. According to Lemma \ref{rims}, there exists $\omega_1>0$ such that, for every $f\prec A_{\omega_1}$ with $\|f\|_\infty\leq\lambda$ we have $|\tilde{V}(f)|\leq \epsilon$.

Since $\tilde{V}$ is continuous at $g$, there exists $\delta>0$ such that, $|\tilde{V}(g)-\tilde{V}(g')|<\epsilon$ for every $g'$ such that $\|g-g'\|_\infty<  \delta$. We define $g'=(g-\frac{\delta}{2})\vee 0$. Then, for every $t\in A$, it follows that
$$
g'(t)=\max\big\{g(t)-\frac{\delta}{2},0\big\}\leq g(t)\leq (f_1\vee f_2)(t)=f_1(t).
$$
Now, we can apply the uniform continuity of $g'$ and $f_1$  to find $\omega_2$ such that for every $t,s\in S^{n-1}$, if $|t-s|<\omega_2$, then $|f_1(t)-f_1(s)|<\delta/4$ and  $|g'(t)-g'(s)|<\delta/4$. In particular, this implies that for every $t\in A_{\omega_2}$, $g'(t)\leq f_1(t)$.

Let $\omega=\min\{\omega_1, \omega_2\}$, and let
$$
J(A, \omega)=A\cup A_\omega
$$
be the open $\omega$-outer parallel of the closed set $A$.
Note that $S^{n-1}=J(A, \omega)\cup B$, where both $J(A, \omega)$ and $B$ are open sets. Moreover, we clearly have $J(A, \omega)\cap B=A_\omega$.

We consider the functions  $\varphi_1\prec J(A,\omega)$, $\varphi_2\prec B$  associated to the decomposition $S^{n-1}=J(A, \omega)\cup B$ by Lemma \ref{split}. Then $\varphi_1\vee \varphi_2=\uno$. Let us define $g'_1=\varphi_1 g'$, $g'_2=\varphi_2 g'$, $h_1=\varphi_1 h$, $h_2=\varphi_2 h$ as in Lemma \ref{split}.

A simple verification yields \begin{itemize}
\item $g'=g'_1\vee g'_2$, $h=h_1\vee h_2$,

\item $g'_1\wedge g'_2\prec A_\omega$, $h_1\wedge h_2\prec A_\omega$,

\item $g'_1\wedge h_2\prec A_\omega$, $h_1\wedge g'_2\prec A_\omega$,

\item $0\leq g'_1\vee h_2 \leq f_1$,

\item $0\leq g'_2\vee h_1\leq  f_2$.

\end{itemize}

Therefore, we get
\begin{eqnarray*}
\tilde{V^+}(f_1\vee f_2)&+& \tilde{V}^+(f_1\wedge f_2) \leq  \tilde{V}(g)+ \tilde{V}(h) + 2\epsilon \leq \tilde{V}(g')+ \tilde{V}(h) + 3\epsilon \\
&=&\tilde{V}(g'_1) + \tilde{V}(g'_2) - \tilde{V}(g'_1\wedge g'_2) + \tilde{V}(h_1) + \tilde{V}(h_2) - \tilde{V}(h_1\wedge h_2)+3\epsilon \\
&\leq &\tilde{V}(g'_1) + \tilde{V}(h_2)+ \tilde{V}(g'_2) +\tilde{V}(h_1)+5\epsilon\\
&=&\tilde{V}(g'_1\vee h_2)+ \tilde{V}(g'_1\wedge h_2)+ \tilde{V}(g'_2\vee h_1)+\tilde{V}(g'_2\wedge  h_1)+ 5\epsilon\\
&\leq & \tilde{V}(g'_1\vee h_2)+  \tilde{V}(g'_2\vee h_1) + 7\epsilon\leq \tilde{V^+}(f_1)+ \tilde{V}^+(f_2)+7\epsilon.
\end{eqnarray*}
Again, since $\epsilon>0$ was arbitrary, this finishes the proof of \eqref{igualdad}.

\smallskip

Let us see now that $\tilde{V}^+$ is continuous. Let us consider $f_0\in C(S^{n-1})^+$ and take $\epsilon>0$. There exists $g_0\in C(S^{n-1})^+$ with $0\leq g_0 \leq f_0$ such that $\tilde{V}^+(f_0)\leq \tilde{V}(g_0)+\epsilon$.

Since $\tilde{V}$ is continuous at $f_0$ and $g_0$,  there exists $\delta>0$ such that for every $f, g \in C(S^{n-1})^+$ with $\|f_0-f\|_\infty<\delta$ and $\|g_0-g\|<\delta$, we have  $|\tilde{V}(f_0)-\tilde{V}(f)|<\epsilon$ and  $|\tilde{V}(g_0)-\tilde{V}(g)|<\epsilon$.

Let now $f\in C(S^{n-1})^+$ be such that $\|f_0-f\|_\infty<\delta$. Pick $g\in C(S^{n-1})^+$ with $0\leq g \leq f$ such that $\tilde{V}^+(f)\leq \tilde{V}(g) + \epsilon$.

Note that  $\|g_0\wedge f - g_0\|<\delta$ and $\|g\vee f_0 - f_0\|<\delta$. Then, we have
$$
\tilde{V}^+(f)\geq \tilde{V}(g_0\wedge f) \geq \tilde{V}(g_0)-\epsilon\geq \tilde{V}^+(f_0)-2\epsilon,
$$
and
\begin{eqnarray*}
\tilde{V}^+(f)&\leq& \tilde{V}(g) + \epsilon =\tilde{V}(g\wedge f_0) + \tilde{V}(g\vee f_0) -  \tilde{V}(f_0) +\epsilon \\
&\leq& \tilde{V}(g\wedge f_0) + |\tilde{V}(g\vee f_0) -  \tilde{V}(f_0) |+ \epsilon \leq \tilde{V}^+(f_0) + 2\epsilon.
\end{eqnarray*}

Hence, $$|\tilde{V}^+(f_0)-\tilde{V}^+(f)|<2\epsilon$$ and $\tilde{V}^+$ is continuous as claimed.

\end{proof}

\begin{remark}
The same proof shows that every continuous ``valuation'' $\tilde{V}:C(K)\longrightarrow \mathbb R$, where $K$ is a metrizable compact space, can be written as a difference of two positive continuous valuations.
\end{remark}

The proof of Corollary \ref{representacion} is now immediate.

\begin{proof}[Proof of Corollary \ref{representacion}]
Let $V:\mathcal S_0^n\longrightarrow \mathbb R$ be a rotationally invariant radial continuous valuation. We decompose it as $V=V^+-V^-$ as in Theorem \ref{main}. According to \cite[Theorem 1.1]{Vi}, there exist two continuous functions  $\theta^+, \theta^-:[0,\infty) \longrightarrow \mathbb R$ such that, for every $K\in \mathcal S_0^n$,
$$V(K)=V^+(K)-V^-(K) =\int_{S^{n-1}} \theta^+(\rho_K(t)) dm(t)- \int_{S^{n-1}} \theta^-(\rho_K(t)) dm(t).$$

We define now $\theta=\theta^+-\theta^-$ and the first part of the result follows.

The converse statement had already been proved in \cite[Theorem 1.1]{Vi} (for that implication, the positivity is not needed).
\end{proof}

\end{document}